\documentclass[12pt,a4paper]{article}
\usepackage{ifthen} 
\usepackage{calc} 

\usepackage[german,english]{babel} 

\usepackage[all]{xy}
\usepackage{amsmath}
\usepackage{amssymb}
\usepackage{latexsym}
\usepackage{enumerate}
\usepackage{theorem}
\usepackage{mathrsfs}
\usepackage{helvet}

\usepackage{lastpage} 
\usepackage{prettyref} 
\usepackage[pagebackref]{hyperref} 

\input xy 

\newboolean{finalversion} 
\setboolean{finalversion}{true}

\theoremheaderfont{\scshape}

\newrefformat{chap}{\hyperref[{#1}]{Chapter~\ref*{#1}}}
\newrefformat{sec}{\hyperref[{#1}]{Section~\ref*{#1}}}

\newtheorem{Notation}{Notation}
\newtheorem{Convention}{Convention}
\newtheorem{Def-Prop}{Definition-Proposition}
\newtheorem{Proof-Expl}{Proof-Explanation}
\newtheorem{theorem}{Theorem}[section]
\newrefformat{thm}{\hyperref[{#1}]{Theorem~\ref*{#1}}}
\newtheorem{definition}[theorem]{Definition}
\newrefformat{def}{\hyperref[{#1}]{Definition~\ref*{#1}}}
\newtheorem{lemma}[theorem]{Lemma}
\newrefformat{lem}{\hyperref[{#1}]{Lemma~\ref*{#1}}}
\newtheorem{proposition}[theorem]{Proposition}
\newrefformat{prop}{\hyperref[{#1}]{Proposition~\ref*{#1}}}

\newrefformat{cor}{\hyperref[{#1}]{Corollary~\ref*{#1}}}
{\theorembodyfont{\upshape}
\newtheorem{examplecore}[theorem]{Example}}
\newrefformat{ex}{\hyperref[{#1}]{Example~\ref*{#1}}}
\newtheorem{facts}[theorem]{Facts}

\newtheorem{remark}[theorem]{Remark}
\newrefformat{rem}{\hyperref[{#1}]{Remark~\ref*{#1}}}

\newcommand{\Id}{\ensuremath{\operatorname{\text{Id}}}}
\newcommand{\Spec}{\ensuremath{\operatorname{Spec}}}

\newcommand{\Hom}{\ensuremath{\operatorname{Hom}}}

\newcommand{\HOM}{\ensuremath{\mathcal Hom}}

\newenvironment{proof}{\noindent\textsc{Proof:}}{\hspace*{\fill}
$\blacksquare$\par\vspace{.1cm}} 

\newenvironment{example}{\begin{examplecore}}{\hspace*{\fill}
$\square$\par\vspace{.1cm}\end{examplecore}}

\newcommand{\mylabel}[1]{\label{#1}\ifthenelse{\boolean{finalversion}}{
  }{\marginpar{\tiny #1}}}  

\setboolean{finalversion}{true}
\begin{document}
\title{ Global algebraic linear differential operators}
\maketitle
\author{Stefan G\"unther}
\begin{abstract} In this note, we want to investigate the question, given a projective algebraic scheme $X/k$\, and coherent sheaves $\mathcal F_1$\, and $\mathcal F_2$\, on $X$, when do global differential operators of order $N>0$\,, $D:\mathcal F_1\longrightarrow \mathcal F_2$\, exist. \\
After remembering the necessary foundational material, we prove one result in this direction, namely  that on $\mathbb P^n_k,\,\, n>1$, for each locally free sheaf $\mathcal E$\, and each $n\in \mathbb Z$\,, there exist  global differential operators $\mathcal E\longrightarrow \mathcal E(n)$\, of order $N$\,, if $N\geq N(n)$\, is sufficiently large. Also, we calculate the order of growth of differential operators of order $\leq N$\,, as $N$\, tends to infinity. As a final result, we prove, that, with the standard definition, for "most" projective algebraic varieties,  algebraic elliptic operators on a locally free sheaf do not exist. 
\end{abstract}

\tableofcontents
\section{Notation and Conventions and Basic Definitions}
\begin{Convention} By $\mathbb N$\, we denote the natural numbers, by $\mathbb N_0$\, the set of nonnegative integers.
\end{Convention}
We use multi index notation: if $x_1,...,x_n$\, is a set of variables, we denote  $\underline{x}^{\underline{m}}: x_1^{m_1}\cdot x_2^{m_2}\ldots x_n^{m_n}$\, where  $\underline{m}:=(m_1,m_2,\ldots ,m_n)$\, is a multi-index of lenght $n$. By $\mid \underline{m}\mid$\, we denot the number $m_1+...+m_n$\,.
The partial derivatives of a function $f(x_1,...,x_m)$\,  in the variables $x_i$\, we denote by 
$\partial^{\mid m\mid}/\partial \underline{x}^{\underline{m}}(f(x_1,...,x_m))$\,.
\begin{Notation} Let $X\longrightarrow S$\, be a morphism of schemes. By $\Omega^{(1)}(X/S)$\, we denote the usual sheaf of K\"ahler differentials. The ideal sheaf of the relative diagonal $\Delta_{X/S}\hookrightarrow X\times_SX$\, we denote by $\mathcal I_{X/S}$\,. Thus $\Omega^{(1)}(X/S)\cong \mathcal I_{X/S}/\mathcal I_{X/S}^2.$
\end{Notation}
We collect some basic facts about jet bundles and differential operators. Proofs and more detailed information can be found in \cite{Guenther2}.\\
\begin{facts} If $q: X\longrightarrow S$ is a morphism of finite type of noetherian schemes and $\mathcal E$\, is a  quasi coherent sheaf on $X$, for each $N\in \mathbb N_0\cup\{\mathbb N\}$\,, we denote by $\mathcal J^N(\mathcal E/S)$\,, the $N^{th}$ jet bundle of $-mathcal E$-, relative to $S$, which is a  quasi coherent sheaf on $X$. If $\mathcal E$\, is coherent, so is $\mathcal J^N(\mathcal E/S)$\, for $N\in \mathbb N_0$\,. -- Locally , if $\Spec B\longrightarrow \Spec A$\,  is the restriction of $q$ to Zariski-open subsets and $\mathcal E$\, restricted to the $\Spec B$\, corresponds to the $B$-module $M$, $\mathcal J^N(\mathcal E/S)$\, is  given  by 
$$\mathcal J^N(M/A)=B\otimes_AM/I_{B/A}\cdot (B\otimes_AM),$$ where $I_{B/A}$\, is the ideal in $B\otimes_AB$\, which is the kernel of the multiplication map. If $\mathcal E'\subset \mathcal E$\,  is a quasi coherent subsheaf, we denote by $\mathcal J^N(\mathcal E'/S)'$\, the image of the homomorphism $\mathcal J^N(i/S)$\,, where $i: \mathcal E'\hookrightarrow \mathcal E$\, is the inclusion. \\
 We denote the universal derivation $\mathcal E\longrightarrow \mathcal J^N(\mathcal E/S)$\, by $d^N_{\mathcal E/S}.$\\ 
If $\mathcal E=\mathcal O_X$\,, the jet sheaf  $\mathcal J^N(\mathcal O_X/S)$\, is an $\mathcal O_X$-algebra and we denote by $J^N(X/S):=\Spec_X\mathcal J^N(\mathcal O_X/S)$\,, the associated affine bundle over $X$ with projection $p_X=p_{1,X}: J^N(X/S)\longrightarrow X$\,. There is a second projection $p_{2,X}: J^N(X/S)\longrightarrow X$\, which corresponds to the universal derivation $d^N_{\mathcal O_X/S}$\,.
\end{facts} 
\begin{proposition}\mylabel{prop: P4}
 \begin{enumerate}[1]
\item Let $X\longrightarrow S$\, be a morphism of finite type between noetherian schemes. For each $N\in \mathbb N\cup\{\infty\}$\,, let $\mathcal J^{N}(-/S)$\, be the functor from quasi coherent $\mathcal O_X$-modules to quasi coherent $\mathcal J^{ N}(X/S)$-modules sending $\mathcal F$ to $\mathcal J^{ N}(\mathcal F/S)$\,. Then, this functor is right  exact and there is a canonical natural isomorphism $\mathcal J^{ N}(-/S)\stackrel{\cong}\longrightarrow p_2^*(-).$
\item If $\pi:X\longrightarrow S$\, is flat, then $\mathcal J^{\mathbb N}(-/S)$\, is an exact functor.
\item If $\pi: X\longrightarrow S$\, is a smooth morphism of schemes, then for each $N\in \mathbb N$\,, the functor $\mathcal J^N(-/S)$, sending quasi coherent $\mathcal O_X$-modules  to quasi coherent $\mathcal J^N(X/S)$-modules , is exact and equal to $(p^N_2)^*$\,.
\end{enumerate}
\end{proposition}
\begin{proof}(see \cite{Guenther2}[section 3.5, Proposition 3.33, p.19].
\end{proof}
 \begin{definition}\mylabel{def:D46} Let $X\longrightarrow S$\, be an arbitrary morphism of schemes, or more generally of algebraic spaces, and  $\mathcal F_i, i=1,2$\, be  quasi coherent sheaves on $X$. Then, a differential operator of order $\leq N$\, is an $\mathcal O_S$-linear map $D: \mathcal F_1\longrightarrow \mathcal F_2$\, that can be factored as
  $\mathcal F_1\stackrel{d^N_{\mathcal F/S}}\longrightarrow \mathcal J^N(\mathcal F_1/S)$\,
  and an $\mathcal O_X$-linear map $\widetilde{D}: \mathcal J^N(\mathcal F_1/S)\longrightarrow \mathcal F_2$\,. \\
  A differential operator of order $N$\, is a differential operator that is of order $\leq N$\, but not of order $\leq N-1$\,.
 \end{definition}
 Thus, in this situation, there is a 1-1 correspondence between differential operators $\mathcal F_1\longrightarrow \mathcal F_2$\, relative to $S$ and $\mathcal O_X$-linear maps $\mathcal J^N(\mathcal F_1/S)\longrightarrow \mathcal F_2$\,. 
\begin{facts} Let $q: X\longrightarrow S$\, be a morphism of finite type of noetherian schemes and $\mathcal F_1,\mathcal F_2$\, be  quasi coherent sheaves on $X$. By $DO^N_{X/S}(\mathcal F_1,\mathcal F_2)$\, we denote the $\mathcal O_S$-module of global differential operators $D: \mathcal F_1\longrightarrow \mathcal F_2$\, relative to $S$, which is by definition the quasi coherent $\mathcal O_X$-module
$$DO^N_{X/S}(\mathcal F_1,\mathcal F_2):=\Hom_{\mathcal O_X}(\mathcal J^N(\mathcal F_1/S), \mathcal F_2).$$
If $\mathcal F_1,\mathcal F_2$\, are coherent, so is $DO^N(\mathcal F_1,\mathcal F_2)$\,.\\
  If $S=\Spec k$\, is a point, we use the simplified notation $DO^N(\mathcal F_1,\mathcal F_2)$\,.
\end{facts}
\begin{facts} Let $q: X\longrightarrow S$\, be a  morphism of finite type between noetherian schemes and $\mathcal E$\, be a coherent sheaf on $X$. Then we denote for each $N\in \mathbb N$\, the standard jet bundle exact sequence by 
$$j^N(\mathcal E/S): \quad 0\longrightarrow \mathcal I_{X/S}^N\mathcal E/\mathcal I_{X/S}^{N+1}\cdot \mathcal E\longrightarrow \mathcal J^N(\mathcal E/S)\longrightarrow \mathcal J^{N-1}(\mathcal E/S)\longrightarrow 0.$$ If $\mathcal E$\, is locally free and the morphism $q: X\longrightarrow S$\, is smooth, the natural homomorphism
\begin{gather*}\mathcal E\otimes_{\mathcal O_X}\Omega^{(1)}(X/S)^{\otimes^s N}\cong \mathcal E\otimes_{\mathcal O_X}\mathcal (I_{X/S}/\mathcal I_{X/S}^2)^{\otimes^s N}\\
\longrightarrow  \mathcal I_{X/S}^{N}/\mathcal I_{X/S}^{N+1}\otimes_{\mathcal O_X}\mathcal E\longrightarrow \mathcal I_{X/S}^N\cdot \mathcal E/\mathcal I_{X/S}^{N+1}\cdot\mathcal E
\end{gather*}
is an isomorphism, which follows from the fact, that in this case $\Delta_{X/S}: X\hookrightarrow X\times_SX$\, is a regular embedding.
\end{facts}
\begin{Notation}If $q: X\longrightarrow S$\, is a smooth morphism of finite type of noetherian schemes  and $\mathcal E_1$\, and $\mathcal E_2$\, are locally free sheaves on $X$,  for each $N\in \mathbb N$\, we denote the standard exact sequence of differential operators by
\begin{gather*} do^N_{X/S}(\mathcal E_1,\mathcal E_2):\quad 0\longrightarrow DO^{N_1}_{X/S}(\mathcal E_1,\mathcal E_2)\longrightarrow DO^{N}_{X/S}(\mathcal E_1,\mathcal E_2)\\
\longrightarrow \Hom_{\mathcal O_X}(\mathcal E_1,\mathcal E_2)\otimes_{\mathcal O_X}\mathcal T^1(X/S)^{\otimes^s N}\longrightarrow 0,
\end{gather*} which is $\Hom_{\mathcal O_X}(j^N(\mathcal E_1/S),\mathcal E_2).$
\end{Notation}
\section{Introduction}
 In this note, we want to investigate the question, given a complete  algebraic projective scheme $X/k$ and a coherent sheaves  $\mathcal E_1,\mathcal E_2$\, on $X$, under which conditions does there exist a differential operator of order $N$,  $D: \mathcal E_1\longrightarrow \mathcal E_2$\,. We in particular treat the case where $X=\mathbb P^n_k, n> 1$\, is projective $n$-space over $k$\,.  One result (see \prettyref{prop:P67}), is that for each $n\in \mathbb Z$\, and each locally free sheaf $\mathcal E$\, on $\mathbb P^n_k$\, there exists for $N>>0$\, a global linear partial differential operator $D: \mathcal E\longrightarrow \mathcal E(n)$\, of order $N$. This is in contrast to the case of $\mathcal O_X$-linear homomorphisms $\phi: \mathcal E\longrightarrow \mathcal E(n)$\, where for $n<< 0$\, thanks to Serre duality no global homomorphisms exist. We have in addition that for fixed $N>0$\, and given nontorsion coherent $\mathcal O_X$-modules $\mathcal F_1,\mathcal F_2$\, there exists a global section $D\in \Gamma(X,DO^N(\mathcal F_1,\mathcal F_2(n)))$\, for $n>>0$\, which does not lie in $\Gamma(X, DO^{N-1}(\mathcal F_1,\mathcal F_2(n))$\, (see \prettyref{lem:L21}). This is a consequence of Serre vanishing and the fact, that the inclusion of coherent sheaves $DO^{N-1}(\mathcal F_1,\mathcal F_2)\subsetneq DO^N(\mathcal F_1,\mathcal F_2)$\, is strict, thanks to the result proven in \cite{Guenther2}[section 3.7, Proposition 3.44, pp. 26-29]. \\ We also calculate for fixed locally free $\mathcal E_1,\mathcal E_2$\, the number of global sections $\Gamma(\mathbb P^n_k, DO^N(\mathcal E_1,\mathcal E_2))$\, as $N$ tends to infinity. \\
 Next, we study the behavior of differential operators $D: \mathcal E\longrightarrow \mathcal E$\ on arbitrary smooth projective $X/k$\, , with respect to the Harder-Narhasimhan filtration  (HN-filtration for short) of $\mathcal E$\,.
 Our main result is \prettyref{prop:P60}, which says, that if $X$\, is not uniruled, then $D$ respects the Harder-Narhasimhan-filtration of $\mathcal E$\,. If moreover, $HN^{\bullet}(\Omega^{(1)}(X/k))$\, is the HN-filtration of the cotangent sheaf and the minimal slope $\mu_{min}(\Omega^{(1)}(X/k))>0$\,, then $D$ respects the HN-filtration and the differential operator $gr^iHN(\mathcal E)\longrightarrow gr^iHN(\mathcal E)$\, is $\mathcal O_{X}$-linear (see \prettyref{prop:P60}).\\
 Finally we study the question of global elliptic differential operators on locally free sheaves. We show that they exist on abelian varieties but on smooth projective varieties with $\mu_{min}(\Omega^{(1)}(X/k))>0$\, they do not exist at all. In \cite{Indexsatz} , an algebraic index theorem has been proved, and the result proved here puts an end to speculuations that there could exist an algebraic index theorem equivalent to the Hirzebruch-Riemann-Roch theorem. 
\section{Global differential operators on coherent sheaves}
If $X/k$ is a complete algebraic scheme, or , more generally a complete algebraic space  over a field $k$\, and $\mathcal F$  is a coherent sheaf on $X$, the question arises, are there global differential operators $\mathcal F\longrightarrow \mathcal F$\,? To get a feeling for this subtle question, we prove a few lemmas and give some examples. 
\begin{lemma}\mylabel{lem:L21} Let $X/k$\, be a projective scheme, $\mathcal F_1,\mathcal F_2$\, be  coherent nontorsion sheaves on $X$ and $\mathcal O_X(1)$\, be an ample invertible sheaf on $X$. Then, for each $N\in \mathbb N$\,  there is $M=M(\mathcal F_1,\mathcal F_2,N)$\, such that for all $m\geq M$\, there is a differential operator of order $N$,  $D:\mathcal F_1\longrightarrow \mathcal F_2(m)$.\\
In particular for $m>>0$\, there always exist non-$\mathcal O_X$-linear operators.
\end{lemma}
\begin{proof} By \cite{Guenther2}[section 3.7, Proposition 3.44, pp. 26-29], for any $N\in \mathbb N$\,, the sheaf 
$$DO^{N}(\mathcal F_1,\mathcal F_2):=\HOM_X(\mathcal J^N(\mathcal F_1/k),\mathcal F_2)$$ is stricly larger then the sheaf $DO^{N-1}_X(\mathcal F_1,\mathcal F_2).$
 There is $M= M(\mathcal F_1,\mathcal F_2)$\, such that  
 \begin{gather*}DO^{N}(\mathcal F_1,\mathcal F_2)(m)=DO^N(\mathcal F_1,\mathcal F_2(m))\quad \text{and}\\
  DO^{N-1}_{\mathcal O_X}(\mathcal F_1,\mathcal F_2)(m)=DO^{N-1}(\mathcal F_1,\mathcal F_2(m))\quad
     \text{and the quotient sheaf}\\
   \mathcal Q(m,N)= DO^N(\mathcal F_1,\mathcal F_2(m))/(DO^{N-1}(\mathcal F_1,\mathcal F_2(m))\\
   =(DO^N(\mathcal F_1,\mathcal F_2)/DO^{N-1}(\mathcal F_1,\mathcal F_2))(m)
   \end{gather*}  are nonzero,  globally generated for $m\geq M$\, and have no higher cohomology.\\
   The sheaf $\mathcal Q(m.N)$\,  is for $\mathcal F_i, i=1,2$\, locally free  and $X/k$\, smooth the sheaf 
   $$ \mathcal Q(m,N):=\HOM_{\mathcal O_X}(\mathcal F,\mathcal F(m))\otimes S^N\mathcal T^1(X/k).$$  Then we get an exact sequence 
 \begin{gather*}0\longrightarrow H^0(X, DO^{N-1}(\mathcal F_1,\mathcal F_2(m)))\longrightarrow H^0(X, DO^N(\mathcal F_1,\mathcal F_2(m)))\\
 \longrightarrow H^0(X, Q(m,N))\longrightarrow 0.
 \end{gather*}
 All cohomology groups are nonzero (because the sheaves are nonzero and globally generated) and thus there exists a global differential operator $D_N\in H^0(X,DO^N(\mathcal F,\mathcal F(m)))$\, not contained in $H^0(X,DO^{N-1}(\mathcal F,\mathcal F(m)))$\,.  Putting $N=0$\, and observing that $DO^0(\mathcal F,\mathcal F(m))=\HOM_{\mathcal O_X}(\mathcal F,\mathcal F(m))$\, we get the last claim. 
 \end{proof}
 \begin{lemma}\mylabel{lem:L108} For each projective  morphism $f:X\longrightarrow S$\, there exists a locally free $\mathcal O_X$-module $\mathcal E$\, and a global differential operator $\mathcal E\longrightarrow \mathcal E$\, of arbitrary high order relative to $S$.
 \end{lemma}
 \begin{proof}  Let $\mathcal O_X(1)$\, be an $f$-ample invertible sheaf. By the previous proposition, there exists differential operators $D_{12}: \mathcal O_X(m)\longrightarrow \mathcal O_X(n+m)$\, for $n>>0$\, of arbitrary high order. Put $\mathcal E:=\mathcal O_X(m)\oplus \mathcal O_X(n+m)$\, and define $D$ to be given by the matrix of differential operators
 \[ D= 
 \begin{pmatrix}  \Id_{\mathcal O_X(m)} & D_{1,2}\\
                  0 & \Id_{\mathcal O_X(m+n)}
                  \end{pmatrix}
                  \]
     Then $D$ is a differential operator on $\mathcal E$\,. The Harder-Narasimhan filtration on $\mathcal E$\, is the filtration $\mathcal HN^{\bullet}\mathcal E$\, with  $\mathcal HN^1(\mathcal E)=\mathcal O_X(n+m)$\, and $gr^2\mathcal HN^{\bullet}(\mathcal E)=\mathcal O_X(m)$\, and  the induced differential operators on the graded pieces are $\mathcal O_X$-linear. By taking direct sums of $\mathcal O_X(n_i)$, $n_i$\, appropriately choosen, we can manage to get the rank of $\mathcal E$\, arbitrarily high.
 \end{proof}
 We have the following 
 \begin{proposition}\mylabel{prop:P30102} Let $q: X\longrightarrow S$\, be a proper morphism of noetherian schemes and $D: \mathcal E_1\longrightarrow \mathcal E_2$\, be a differential operator of coherent sheaves relative to $S$ of some order $\leq N$\,. Then, there are $\mathcal O_S$-linear maps 
 $$R^iD: R^iq_*\mathcal E_1\longrightarrow R^iq_*\mathcal E_2.$$
 \end{proposition}
 \begin{proof} There is the standard factorization of $D$ as 
 $$\mathcal E_1\stackrel{d^N_{\mathcal E_1/S}}\longrightarrow \mathcal J^N(\mathcal E_1/S)\stackrel{\widetilde{D}}\longrightarrow \mathcal E_2.$$
 The first map $d^N_{\mathcal E_1/S}$\, is $\mathcal O_X$-linear if the $\mathcal O_X$-bimodule $\mathcal J^N(X/S)$\, is considered as a coherent  $\mathcal O_X$-module via the second $\mathcal O_X$-module structure. For a  detailed  discussion of this, see \cite{Guenther2}[section 3.5, pp. 16-17, Lemma 3.27].  Hence, we get an $\mathcal O_S$-linear map 
 $$R^iq_*\mathcal E_1\stackrel{R^iq_*d^N_{\mathcal E_1/S}}\longrightarrow R^iq_*(\mathcal J^N(\mathcal E_1/S)^{(2)}).$$ Regarding $\mathcal J^N(\mathcal E_1/S)$\, with its first $\mathcal O_X$-module-structure, we get an $\mathcal O_S$-linear map 
 $$R^iq_*\widetilde{D}: R^iq_*(\mathcal J^N(\mathcal E_1/S)^{(1)})\longrightarrow R^iq_*\mathcal E_2.$$
 By \cite{Guenther2}[section 3.5, pp. 16-17, Lemma 3.27], we have 
 $$\mathcal J^N(\mathcal E_1/S)=p_{1*}(\overline{\mathcal O_S\otimes_{\mathcal O_S}\mathcal E_1)}= p_{2,*}\overline{\mathcal O_X\otimes_{\mathcal O_S}\mathcal E_1},$$
  where $p_1,p_2: J^N(X/S)\longrightarrow X$\, are the two projection morphisms, which are morphisms of $S$-schemes and where $J^N(X/S)=\Spec_X\mathcal J^N(X/S)$. 
 Since $q\circ p_1=q\circ p_2$\,  and $p_1$ and $p_2$\, are affine, we have 
 \begin{gather*}
 R^iq_*\mathcal J^N(\mathcal E_1/S)^{(2)}=R^i(q\circ p_2)_*\overline{\mathcal O_X\otimes_{\mathcal O_S}\mathcal E_1}\\
 =R^i(q\circ p_1)_*\overline{\mathcal O_X\otimes_{\mathcal O_S}\mathcal E_1} =R^iq_*\mathcal J^N(\mathcal E_1/S)^{(1)}.
 \end{gather*}
 Thus, we can compose $R^iq_*d^N_{\mathcal E_1/S}$\, with the map $R^iq_*{\widetilde{D}}$\, to get the  required $\mathcal O_S$-linear map 
 $R^iq_*D: R^iq_*\mathcal E_1\longrightarrow R^iq_*\mathcal E_2.$
\end{proof}
\subsection{Extensions of differential operators}
The basic question is the following: If 
$$0\longrightarrow \mathcal E_1\longrightarrow \mathcal E\longrightarrow \mathcal E_2\longrightarrow 0$$
is an exact sequence of coherent sheaves on an $S$-scheme $X/S$\, and $D_1:\mathcal E_1\longrightarrow \mathcal E_1$\, and $D_2:\mathcal E_2\longrightarrow \mathcal E_2$\, are differential operators on $\mathcal E_1,\mathcal E_2$\,, respectively, relative to $S$, can one find a differential operator $D:\mathcal E\longrightarrow \mathcal E$\, on $\mathcal E$\, fitting the exact sequence. We  want to show that this is not always the case.
\begin{proposition}\mylabel{prop:P20} Let $X/S$\, be an integral scheme of finite type and 
$$0\longrightarrow \mathcal E_1\longrightarrow \mathcal E\longrightarrow T\longrightarrow 0$$ be an exact sequence, where  $\mathcal E_1,\mathcal E$\, are locally free and $T$ is a torsion sheaf. Given a differential operator $D_1:\mathcal E_1\longrightarrow \mathcal E_1$, there is at most one differential operator $D:\mathcal E\longrightarrow \mathcal E$\,, extending $D_1$\,.
\end{proposition}
\begin{proof} Suppose, $D_1$\, is of order $N$\, and there  are two differential operators $D,D'$\, of order $\leq N'$\, with $N'\geq N$\, on $\mathcal E$\, extending $D_1$\,. Without loss of generality, we may assume that $N=N'$\,. Considering the difference $D-D'$\,, we may assume that $D_1=0$\,. We look for an $\mathcal O_X$-linear homomorpism 
$$\widetilde{D}: \mathcal J^N(\mathcal E/k)\longrightarrow \mathcal E.$$ By assumption, $\widetilde{D}$\, restricted to $\mathcal J^N(\mathcal E_1/k)'\subset \mathcal J^N(\mathcal E/k)$\, is the zero homomorphism. But, by the right exactness of the jet module functor, the quotient module is the jet-module $\mathcal J^N(\mathcal E/\mathcal E_1/S)$\,. By \cite{Guenther2}[section 3.5, Lemma 3.30, p.18], this is a torsion sheaf, since $\mathcal E/\mathcal E_1\cong T$\, is so.
Thus $\widetilde{D}$\, factors through $\mathcal J^N(\mathcal E/k)/\mathcal J^N(\mathcal E_1/k)'\longrightarrow \mathcal E$\,. But the first sheaf is a torsion $\mathcal O_X$-module so $\widetilde{D}$\, must be the zero homomorphism.
\end{proof}
\begin{example} Let $(X,\mathcal O_X(H))$\, be a polarized projective scheme with $\mathcal O_X(H)$\, very ample and $H=\text{div}(s), s\in \Gamma(X,\mathcal O_X(H))$\, be a smooth section. Let $n\in \mathbb N$\, be choosen such that there exists a global differential operator $D_{12}: \mathcal O_X\longrightarrow \mathcal O_X(nH)$\,.
Consider the extension of locally free sheaves  
$$0\longrightarrow \mathcal O_X\oplus \mathcal O_X(nH)\longrightarrow \mathcal O_X(H)\oplus \mathcal O_X((n+1)H)\longrightarrow \mathcal O_H(H)\oplus \mathcal O_H((n+1)H)\longrightarrow 0.$$ The global  $\mathcal O_X$-($\mathcal O_H$-linear) endomorphisms of the last sheaf contain the endomorphisms in diagonal form and thus 
$$h^0(H, \Hom_H(\mathcal O_H(H)\oplus \mathcal O_H((n+1)H, \mathcal O_H(H)\oplus \mathcal O_H((n+1)H)))\geq 2.$$ 
By the previous proposition, the differential operator $D: \mathcal O_X\oplus \mathcal O_X(nH)\longrightarrow \mathcal O_X\oplus \mathcal O_X(nH)$\, constructed in \prettyref{lem:L108} has at most one extension to an operator $E$\, on $\mathcal O_X(H)\oplus \mathcal O_X((n+1)H)$\, so not every pair $(D, \phi)$\, where $\phi$\, is an $\mathcal O_H$-linear map on the right hand side extends.
\end{example}  

\subsection{Differential operators and semistable sheaves}
Let $X/k$\, be a projective  algebraic scheme of dimension $n$ and $\mathcal O_X(H)=\mathcal O_X(1)$\, be an ample invertibel sheaf on $X$. Let $\mathcal E$\, be a $\mu$-semi-stable sheaf on $X$ with respect to the polarization $\mathcal O_X(1)$\,. We want to investigate the question under which circumstances there exists a global differential operator $\mathcal E\longrightarrow \mathcal E$\,. We first give a criterion, when the answer is always negative.
\begin{proposition}\mylabel{prop:P13} Let $X/k, \mathcal E,\, \mathcal O_X(H)$\, be as above and suppose, that for some $m_0\in \mathbb N$\,, there are effective  Cartier divisors $D_i, i=1,...,N$\, with $D_i\cdot H^{n-1}>0$\, and a homomorphism with torsion kernel 
$$\bigoplus_{i=1}^{M(m)}\mathcal O_X(D_i)\longrightarrow \Omega^{(1)}(X/k)^{\otimes^s m}\,\,\forall m\geq m_0$$ Then , for every torsion free  coherent semistable sheaf $\mathcal E$,  every global differential operator $D: \mathcal E\longrightarrow \mathcal E$\,  has order $< m_0$\,. In particular if $m_0=1$\,, then every differential operator on $\mathcal E$\, is $\mathcal O_X$-linear. If $\mathcal E$\, is arbitrary torsion free, then each $D$ respects the $HN$-filtration and the differential operators on the graded pieces are $\mathcal O_X$-linear.
\end{proposition}
\begin{proof} Let $\mathcal E$\, be a torsion free semistable sheaf on $X$ and $D: \mathcal E\longrightarrow \mathcal E$\, be a differential operator corresponding to an $\mathcal O_X$-linear map: $\widetilde{D}: \mathcal J^N(\mathcal E/k)\longrightarrow \mathcal E$\, with $N$ minimal  with $N>m_0$\,. From the jet bundle exact sequence  $j^N_{X/S}(\mathcal E_1)$, we get a homomorphism of coherent sheaves
$$(\bigoplus_{i=1}^N\mathcal O_X(D_i))\otimes_{\mathcal O_X} \mathcal E\longrightarrow \Omega^{(1)}(X/k)^{\otimes_s^N}\otimes_{\mathcal O_X}\mathcal E\stackrel{f}\longrightarrow \mathcal E.$$
If $f$ is not the zero homomorphism, there must be a direct summand $\mathcal E(D_i)$\, such that the above homomorphism restricted to $\mathcal E(D_i)$\, gives a nonzero homomorphism $\mathcal E(D_i)\longrightarrow \mathcal E$\,. Then $\mathcal E(D_i)$\, is semistable and 
$$\mu(\mathcal E(D_i))=\mu(E)+D_i\cdot H^{n-1}>\mu(\mathcal E).$$
 By \cite{Huybrechts-Lehn}[chapter 1, Proposition 1.2.7, p.11], this homomorphism must be zero. Thus, $\widetilde{D}$\, restricted to $\Omega^{(1)}(X/k)^{\otimes^s N}\otimes \mathcal E$\, is zero. By the same standard exact sequences for the jet modules $j^N_{X/S}(\mathcal E)$, $\widetilde{D}$\, factors over $\mathcal J^{N-1}(\mathcal E/k)$\,. This is in contradiction with the assumed minimality of $N$. Thus we must have had $N\leq m_0$\,.\\
To the last point, if $m_0=1$\, and $\mathcal E$\, is arbitrary torsion free, let $D: \mathcal E\longrightarrow \mathcal E$\, and  the Harder Narhasimhan filtration $HN^{\bullet}(\mathcal E)$ be given. Let $j$ be minimal such that $D(HN^1(\mathcal E))\subset HN^j(\mathcal E)$\, so we get $\overline{D}:gr^1HN(\mathcal E)\longrightarrow gr^jHN(\mathcal E)$\,, since by \cite{Guenther2}[section 3.6 Lemma 3.43, p.24], differential operators restricted to subsheaves and quotient sheaves are again differential operators). Argueing as above, we see that $j$ must be equal to $1$ and by the first part of the proposition, $\overline{D}$\, must be $\mathcal O_X$-linear. So we know that $D(HN^1(\mathcal E))\subseteq HN^1(\mathcal E)$. We get a differential operator $D': \mathcal E/HN^1(\mathcal E)\longrightarrow \mathcal E/HN^1(\mathcal E)$\, and we can continue by induction on the lenght of the $HN$-filtration, the start of the induction given by the first part of the proposition.
\end{proof}
\begin{proposition}\mylabel{prop:P14} Let $X/k$\, be a smooth projective algebraic scheme, $\mathcal O_X(H)$\, be a very ample invertible sheaf on $X$ and $\mathcal E$\, be a coherent sheaf on $X$. Suppose, $\Omega^{(1)}(X/k)$\, is globally generated, e.g., $X$ is an abelian variety. Then each differential operator $D: \mathcal E\longrightarrow \mathcal E$\, respects the Harder-Narasimhan-Filtration $HS^i(\mathcal E)$\,.
\end{proposition}
\begin{proof} Let $D:\mathcal E\longrightarrow \mathcal E$\, and  the Harder-Narhasimhan filtration $HN^{\bullet}(\mathcal E)$\, be given. Let $j$ be minimal such that $D(HN^1(\mathcal E))\subseteq HN^j(\mathcal E)$\,. Then, we get a nonzero differential operator $\overline{D}: HN^1(\mathcal E)\longrightarrow gr^jHN(\mathcal E)$\,.
If $\Omega^{(1)}(X/k)$\, is globally generated, so is each symmetric power. Let  $N\in \mathbb N$\, be minimal such that $\overline{D}$ factors over $\mathcal J^N(\mathcal HN^1(\mathcal E)/k)$. We get homomorphisms of $\mathcal O_X$-modules
$$(\bigoplus \mathcal O_X)\otimes HN^1\mathcal E\twoheadrightarrow \Omega^{(1)}(X/k)^{\otimes^sN}\otimes_{\mathcal O_X}\mathcal E\longrightarrow gr^jHN(\mathcal E),$$
Restricting to each single direct summand $\mathcal O_X\otimes HN^1(\mathcal E)\cong HN^1(\mathcal E)\longrightarrow gr^jHN(\mathcal E)$\, the standard arguement shows, that this map must be zero unless $j=1$. Since this homomorphism is then zero for each direct summand,  so is the homomorphism
$$\Omega^{(1)}(X/k)^{\otimes^s N}\otimes_{\mathcal O_X}HN^1(\mathcal E)\longrightarrow gr^jHN(\mathcal E)$$  But then $\overline{D}$\, factors through $\mathcal J^{N-1}(HN^1(\mathcal E)/k)$\,, a contradiction. Thus, $j=1$\,, we take the quotient differential operator $D_1: \mathcal E/HN^1(\mathcal E)\longrightarrow \mathcal E/HN^1(\mathcal E)$\, and go on by induction.
\end{proof}
\begin{remark}\mylabel{rem:R2412} The same arguement shows that it suffices to assume that $\Omega^{(1)}(X/k)$\, is almost globally generated, i.e. that there is a homomorphism $\bigoplus \mathcal O_X\longrightarrow \Omega^{(1)}(X/k)$\,, with torsion cokernel.
\end{remark}
\begin{proposition}\mylabel{prop:P60} Let $(X,\mathcal O_X(1))$\, be a  smooth polarized projective variety over a field $k$ of characteristic zero and  suppose that $\mu_{min}(\Omega^{(1)}(X/k))\geq 0$\,. If $\mathcal E$\, is an arbitrary torsion free sheaf on $X$, then each differential operator $D:\mathcal E\longrightarrow \mathcal E$\, respects the $HN$-filtration, and  if $\mu_{min}(\Omega^{(1)}(X/k))>0$\,, the differential operator $\overline{D}_i: gr^iHN(\mathcal E)\longrightarrow gr^iHN(\mathcal E)$\, is $\mathcal O_X$-linear.
\end{proposition}
\begin{proof} The arguement is similar to the proof of the previous proposition.  Observe first, that if $\mu_{min}(\Omega^{(1)}(X/k))>0$\, then also $\mu_{min}(\Omega^{(1)}(X/k)^{\otimes^s m})>0$\,. Here we need that the characteristic of $k$\, is zero, because we need that the symmetric tensor product of semi-stable sheaves is again semistable. Now,  let $\mathcal E$\, be torsion free and $D: \mathcal E\longrightarrow \mathcal E$\,be given.  Let again $j$\, be minimal such that $D(HN^1(\mathcal E))\subset HN^j(\mathcal E)$\, so we get $\overline{D}: HN^1(\mathcal E)\longrightarrow gr^jHN(\mathcal E)$\,. Let $N\in \mathbb N$\, be minimal such that $\overline{D}$ factors over $\mathcal J^N(\mathcal HN^1(E)/k)$\,. We then get a nonzero homomorphism
$$\Omega^{(1)}(X/k)^{\otimes^s N}\otimes_{\mathcal O_X}HN^1(\mathcal E)\longrightarrow gr^jHN(\mathcal E).$$
We show by induction that this homomorphism is zero if $j>1$\, and $N>1$\,. Let $HN^{\bullet}(\Omega^{(1)}(X/k)^{\otimes^s N})$\, be the Harder-Narasimhan-filtration. Let $\Omega_{1,N}=HN^1$\, be the maximal destabilizing subsheaf. We get by restriction a homomorphism
$$\Omega_{1,N}\otimes HN^1(\mathcal E)\longrightarrow gr^j\mathcal HN^{\bullet}(\mathcal E).$$
$\Omega_{1,N}\otimes HN^1(\mathcal E)$ is semistable and 
\begin{gather*}\mu(\Omega_{1,N}\otimes HN^1(\mathcal E))=\mu(\Omega_{1,N})+\mu(\mathcal HN^1(\mathcal E))\\
>\mu_{min}(\Omega^{(1)}(X/k)^{\otimes^s N})+\mu(HN^1(\mathcal E))>\mu_j(\mathcal E),
\end{gather*} so by \cite{Huybrechts-Lehn}[chapter 1, Proposition 1.2.7, p.11], this homomorphism is zero. At this point, we also need that the characteristic is zero since we need the semi stability  of the tensor product. Suppose that for some $n\in \mathbb N$\, the restricted homomorphism
$$HN^n(\Omega^{(1)}(X/k)^{\otimes^s N})\otimes HN^1(\mathcal E)\longrightarrow gr^jHN(\mathcal E)$$ is nonzero. We then get a homomorphsim
$$\Omega_{n+1,N}\otimes \mathcal HN^1(\mathcal E):=gr^{n+1}HN^{\bullet}(\Omega^{(1)}(X/k)^{\otimes^s N})\otimes HN^1(\mathcal E)\longrightarrow gr^jHN(\mathcal E)$$
and the same as the previous arguement shows, since $\mu(\Omega_{n+1,N})>\mu_{min}>0$\, that this map must be zero if $j>1$. Thus, the entire map $\Omega^{(1)}(X/k)^{\otimes^s N}\otimes HN^1(\mathcal E)\longrightarrow gr^jHN(\mathcal E)$\, is zero. Since $N$ was choosen minimal, we get by the $N^{th}$ exact sequence of jet bundles a contradiction. Thus $j=1$\,. We get a quotient differential operator $D_1: \mathcal E/HN^1(\mathcal E)\longrightarrow \mathcal E/HN^1(\mathcal E)$\, and we argue by induction on the lenght of the HN-filtration.\\
 To the last point, if $\mathcal E$\, is semistable and a differential operator $D:\mathcal E\longrightarrow \mathcal E$\, is given, choose again $N$ minimal such that $D$ factors over $\mathcal J^N(\mathcal E/k)$\, and show, as in the first part of this proof, that the homomorphism
$$\Omega^{(1)}(X/k)^{\otimes^s N}\otimes \mathcal E\longrightarrow \mathcal E$$
is the zero map if $N>1$\,.
\end{proof}
\begin{remark}\mylabel{rem:R200} By \cite{Miyaoka}[Lecture III, 2.14 Theorem, p. 67], the condition that \\
 $\mu_{min}(\Omega^{(1)}(X/k))\geq 0$\, is equivalent to the uniruledness of $X$. 
\end{remark}
\subsection{Vector bundles with global differential operators on projective $n$-space}
\begin{proposition}\mylabel{prop:P67} Let $X=\mathbb P_k^n$\, be projective $n$-space with $n\geq 2$. Then for each  pair of locally free sheaves $\mathcal F_i, i= 1,2,$ there exists an $N=N(\mathcal F_i)$\, such that for each $N\geq N(\mathcal F_i)$ there exist   operators $D: \mathcal F_1\longrightarrow \mathcal F_2$\,  of order $\geq N$\,.
\end{proposition}
\begin{proof} The tangent sheaf $\mathcal T^1(\mathbb P^n_k/k)$\, is ample . For each $N\in \mathbb N$, we consider the exact sequence 
$$j^N(\mathcal F_1): 0\longrightarrow \Omega^{(1)}(X/k)^{\otimes^s N}\otimes_{\mathcal O_X}\mathcal F_1\longrightarrow \mathcal J^N(\mathcal F_1/k)\longrightarrow \mathcal J^{N-1}(\mathcal F_1/k)\longrightarrow 0.$$ Taking duals and tensoring with $\mathcal F_2$\, we get exact sequences
\begin{gather*}do^N(\mathcal F_1,\mathcal F_2): 0\longrightarrow DO^{N-1}(\mathcal F_1,\mathcal F_2)\stackrel{i_N}\longrightarrow DO^N(\mathcal F_1,\mathcal F_2)\\
\stackrel{\sigma^N}
\longrightarrow \mathcal T^1(\mathbb P^n_k/k)^{\otimes^s N}\otimes Hom_{\mathcal O_X}(\mathcal F_1,\mathcal F_2)\longrightarrow 0.
\end{gather*}
We have $DO^0(\mathcal F_1,\mathcal F_2)\cong Hom_{\mathcal O_X}(\mathcal F_1,\mathcal F_2)$\,. The claim is that for $N>>0$\, 
$$H^0(\mathbb P^n_k, DO^N(\mathcal F_1,\mathcal F_2))\supsetneq H^0(\mathbb P^n_k, DO^{N-1}(\mathcal F_1,\mathcal F_2)\supsetneq H^0(\mathbb P^n_k, Hom_{\mathcal O}(\mathcal F_1,\mathcal F_2)).$$
By \cite{Lazarsfeld}[II, chapter 6.1.B, Theorem 6.1.10, p.11] for each coherent sheaf $\mathcal E$\, on $\mathbb P^n_k$\,  and for all $m\geq m(\mathcal E)$\, we have $H^i(\mathbb P^n_k, S^m\mathcal T^1(\mathbb P^n_k)\otimes \mathcal E)=0\,\,\forall i>0$\, and the sheaf $S^m\mathcal T^1(\mathbb P^n_k)\otimes \mathcal E$\, is globally generated. Putting $\mathcal E=Hom_{\mathcal O_X}(\mathcal F_1,\mathcal F_2)$ and $i=1$\, we get 
$$H^1(\mathbb P^n_k, S^m\mathcal T^1(\mathbb P^n_k/k)\otimes Hom_{\mathcal O_X}(\mathcal F_1,\mathcal F_2))=0\,\, \forall m\geq m(\mathcal E).$$
Writing out the long exact cohomology sequences for the exact sequence $do^N(\mathcal F_1,\mathcal F_2), N\geq m(\mathcal E)$\,, we get
\begin{gather*} 0\longrightarrow H^0(\mathbb P^n_k, DO^{N-1}(\mathcal F_1,\mathcal F_2))\longrightarrow H^0(\mathbb P^n_k, DO^N(\mathcal F_1,\mathcal F_2))\\
\longrightarrow H^0(\mathbb P^n_k, S^N\mathcal T^1(\mathbb P^n_k/k)\otimes Hom_{\mathcal O_X}(\mathcal F_1,\mathcal F_2))\\
\longrightarrow H^1(\mathbb P^n_k, DO^{N-1}(\mathcal F_1,\mathcal F_2)))\longrightarrow  H^1(\mathbb P^n_k, DO^N(\mathcal F_1,\mathcal F_2))\\
\longrightarrow H^1(\mathbb P^n_k, S^N\mathcal T^1(\mathbb P^n_k/k)\otimes Hom_{\mathcal O_X}(\mathcal F_1,\mathcal F_2))=0.
\end{gather*}
 The $k$-linear maps $H^1(\mathbb P^n_k, i_N)$\, are for $N>m(\mathcal F_i)$\, surjective. Thus, there must be an $N(\mathcal F_i)\geq m(\mathcal F_i)$\, such that $H^1(\mathbb P^n_k,i_N)$\, are isomorphisms for $N\geq N(\mathcal F_i)$\,. Looking into the long exact cohomology sequence, we see that for such $N$, the maps 
 $$H^0(\mathbb P^n_k, DO^N(\mathcal F_1,\mathcal F_2))\longrightarrow H^0(\mathbb P^n_k, S^N\mathcal T^1(\mathbb P^n_k/k)\otimes_{\mathcal O_X}Hom_{\mathcal O_X}(\mathcal F_1,\mathcal F_2))\longrightarrow 0$$ are surjective.  The $k$-vector space dimension $h^0(\mathbb P^n_k, S^N\mathcal T^1(\mathbb P^n_k/k))$\, is always greater than zero for $N>>0$\,, because by \cite{Lazarsfeld}[II, chapter 6.1.B, Theorem 6.1.10, p.11] the sheaf $S^N\mathcal T^1(\mathbb P^n_k/k)\otimes Hom_{\mathcal O_X}(\mathcal F_1,\mathcal F_2)$\, is globally generated for $N>>0$\,. So there are elements in $H^0(\mathbb P^n_k, DO^N(\mathcal F_1,\mathcal F_2))$\, that do not lie in $H^0(\mathbb P^n_k, DO^{N-1}(\mathcal F_1,\mathcal F_2))$\,. So for $N>>0$\,, we have strict inclusions 
 $$H^0(\mathbb P^n_k, \Hom_{\mathcal O_X}(\mathcal F_1,\mathcal F_2))\subsetneq H^0(\mathbb P^n_k, DO^{N-1}(\mathcal F_1,\mathcal F_2))\subsetneq H^0(\mathbb P^n_k, DO^{N}(\mathcal F_1,\mathcal F_2)).$$ 
\end{proof}
\begin{remark}\mylabel{rem:R100} By global generation  of the sheaf $DO^N(\mathcal F_1,\mathcal F_2)$\, after a sufficiently high ample twist, on gets immediately differential operators $\mathcal E\longrightarrow \mathcal E(n)$\, for $n>>0$\,. This result implies, that for each $n>0$\, there exist differential operatos of arbitrary high degree $\mathcal E\longrightarrow \mathcal E(-n)$\,.
\end{remark}
 Using the ampleness of $\mathcal T^1(\mathbb P^n/k)$ we  now want to  estimate the growth of $DO^N(\mathcal F_1,\mathcal F_2)$\, as $N$\, tends to infinity.\\
 Let $\mathcal E_1$\, and $\mathcal E_2$\, be locally free sheaves on $\mathbb P^n_k$\,, $n\geq 2$\, be given. By the ampleness of $\mathcal T^1_{\mathbb P^n_k/k}$\,, we know that for some $N\in \mathbb N_0$\, and all $d\geq N$, we have that $H^0(\mathbb P^n_k, S^d\mathcal T^1(\mathbb P^n/k)\otimes \Hom_{\mathcal O}(\mathcal E_1,\mathcal E_2))$\, is globally generated and all higher cohomology groups vanish. As we have shown, there is $M\in \mathbb N$\,, that only depends on $\mathcal E_1,\mathcal E_2$\, such that the sequence
 \begin{gather*}0\longrightarrow H^0(\mathbb P^n_k, DO^{d-1}(\mathcal E_1,\mathcal E_2))\longrightarrow H^0(\mathbb P^n_k, DO^d(\mathcal E_1,\mathcal E_2))\\
 \longrightarrow H^0(\mathbb P^n_k, S^d\mathcal T^1\otimes \Hom_{\mathcal O_{\mathbb P^n_k}}(\mathcal E_1,\mathcal E_2))\longrightarrow 0
 \end{gather*}
 is exact for all $d\geq M$\,. For $d>>M$\, we can thus write for $h^0(\mathbb P^n_k, DO^d(\mathcal E_1,\mathcal E_2))$\, 
 \begin{gather*}h^0(\mathbb P^n_k, DO^d(\mathcal E_1,\mathcal E_2))= h^0(\mathbb P^n_k, DO^M(\mathcal F_1,\mathcal F_2))\\
 + \sum_{M+1\leq k\leq d}h^0(\mathbb P^n_k, S^k\mathcal T^1(\mathbb P^n_k/k)\otimes \Hom_{\mathcal O_{\mathbb P^n}}(\mathcal E_1,\mathcal E_2))=\\
  h^0(\mathbb P^n_k, DO^M(\mathcal F_1,\mathcal F_2))+\sum_{M+1\leq k\leq d}\chi(\mathbb P^n_k,S^k\mathcal T^1(\mathbb P^n_k/k)\otimes \Hom_{\mathcal O}(\mathcal E_1,\mathcal E_2)).
 \end{gather*} 
 If we tensor the symmetric power of the Euler sequence ($\mathbb P^n_k=\mathbb P(V)$)
 $$0\longrightarrow S^{k-1}V\otimes \mathcal O_{\mathbb P}(-1)\longrightarrow S^k(V)\otimes \mathcal O_{\mathbb P}\longrightarrow S^k\mathcal T^1(\mathbb P^n/k)(-k)\longrightarrow 0,\quad\text{or}$$
 $$0\longrightarrow S^{k-1}V\otimes \mathcal O(k-1)\longrightarrow S^k(V)\otimes \mathcal O(k)\longrightarrow S^k\mathcal T^1(\mathbb P^n_k/k)\longrightarrow 0$$
 (see \cite{Huybrechts-Lehn}[chapter 1.4, p. 19]) with $\Hom_{\mathcal O}(\mathcal E_1,\mathcal E_2)$\, and use the additivity of the Euler characteristic, we get 
 \begin{gather*}\chi(\mathbb P^n_k, S^k\mathcal T^1(\mathbb P^n_k/k)\otimes_{\mathcal O}\Hom_{\mathcal O}(\mathcal E_1,\mathcal E_2))=\chi(\mathbb P^n_k, S^k(V)\otimes \mathcal O(k)\otimes \Hom_{\mathcal O}(\mathcal E_1,\mathcal E_2))\\
 - \chi(\mathbb P^n_k, S^{k-1}V\otimes \mathcal O(k-1)\otimes \Hom_{\mathcal O}(\mathcal E_1,\mathcal E_2))=\\
 \binom{n+k}{k}\cdot \chi(\mathbb P^n_k, \Hom_{\mathcal O}(\mathcal E_1,\mathcal E_2(k)))
 - \binom{n+k-1}{k-1}\cdot \chi(\mathbb P^n, \Hom_{\mathcal O}(\mathcal E_1,\mathcal E_2(k-1))).
 \end{gather*}
  Summing over $M+1\leq k\leq d$\, we get for $d>>M$\, 
 \begin{gather*}h^0(\mathbb P^n_k, DO^d(\mathcal E_1,\mathcal E_2))=h^0(\mathbb P^n_k, DO^M(\mathcal E_1,\mathcal E_2))\\
 +\sum_{M+1\leq k\leq d}(\binom{n+k}{k}\cdot \chi(\mathbb P^n_k, \Hom_{\mathcal O}(\mathcal E_1,\mathcal E_2(k)))\\
 -\binom{n+k-1}{k-1}\cdot\chi(\mathbb P^n_k, \Hom_{\mathcal O}(\mathcal E_1,\mathcal E_2(k-1))))\\
 =
 h^0(\mathbb P^n_k, DO^M(\mathcal E_1,\mathcal E_2))\\
 +\binom{n+d}{d}\cdot \chi(\mathbb P^n_k, \Hom_{\mathcal O}(\mathcal E_1,\mathcal E_2(d)))-\binom{n+M}{M}\cdot\chi(\mathbb P^n_k, \Hom_{\mathcal O}(\mathcal E_1,\mathcal E_2(M))).
 \end{gather*}
 We thus have proved the following 
  \begin{proposition}\mylabel{prop:P2612} Let $\mathcal E_1,\mathcal E_2$ be locally free sheaves on $\mathbb P^n_k, n\geq 2$\, . Let $F(N): = h^0(\mathbb P^n_k, DO^N(\mathcal E_1,\mathcal E_2)).$ There is a polynomial of degree $2n$, $P(N) =P_{(\mathcal E_1,\mathcal E_2)}(N)$\, such that for $N>>0$\,  
 $$F(N)= P(N).$$ 
 The polynomial $P(N)$\, is  explicitely given by 
 \begin{gather*}P(N):= \binom{n+N}{N}\cdot \chi(\mathbb P^n_k, \Hom_{\mathcal O}(\mathcal E_1,\mathcal E_2)(N))+\\
 h^0(\mathbb P^n_k, DO^M(\mathcal E_1,\mathcal E_2)) - \binom{n+M}{M}\cdot \chi(\mathbb P^n_k, \Hom_{\mathcal O}(\mathcal E_1,\mathcal E_2)(M)),
 \end{gather*}
 where $M=M(\mathcal E_1,\mathcal E_2)$\, is a fixed natural number.
 \end{proposition}
 \begin{proof}
 \end{proof} 
 For the special case where $\mathcal E_1=\mathcal O(m_1)$\, and $\mathcal E_2=\mathcal O(m_2)$\,, the formula then reads, 
 \begin{gather*}h^0(\mathbb P^n_k, DO^d(\mathcal O(m_1),\mathcal O(m_2)))= h^0(\mathbb P_k^n, DO^M(\mathcal O(m_1), \mathcal O(m_2)))\\
 + \binom{n+d}{d}\cdot \chi(\mathbb P^n, \mathcal O(m_2-m_1+d))-\binom{n+M}{M}\cdot \chi(\mathbb P^n, \mathcal O(m_2-m_1+M)).
 \end{gather*}
 \subsection{Elliptic operators in algebraic geometry}
\begin{definition}\mylabel{def:D13113}(Definition of the symbol of a differential operator)\\
Let $X/k$\, be a smooth variety and $D: \mathcal E_1\longrightarrow \mathcal E_2$\, be a differential operator of order $N\in \mathbb N$\,. Consider the standard exact sequence
\begin{gather*}do^N(\mathcal E_1,\mathcal E_2): 0\longrightarrow DO^{N-1}_{(X/k)}(\mathcal E_1,\mathcal E_2)\longrightarrow DO^N_{(X/k}(\mathcal E_1,\mathcal E_2)\\
\stackrel{\sigma_N}\longrightarrow Hom_{\mathcal O_X}(\mathcal E_1,\mathcal E_2)\otimes_{\mathcal O_X}S^N\mathcal T^1(X/k).
\end{gather*}
The element $\sigma^N(D)\in Hom_{\mathcal O_X}(\mathcal E_1,\mathcal E_2)\otimes_{\mathcal O_X}S^N\mathcal T^1(X/k)$\, is called the symbol of the differential operator $D$.
\end{definition}
Recall, that if $\mathcal F$\, is a coherent sheaf on $X$ and $F: =\Spec_X\text{Sym}^{\bullet}(\mathcal F)$\, with projection $\pi: F\longrightarrow X$\, is the associated affine bundle, for each $N\in \mathbb N$\,, there is a canonical homomorphism
$c_{N,\mathcal F}:\quad\pi^*S^N\mathcal F\longrightarrow \mathcal O_F$\,.  Zariski locally, if $X=\Spec A$\, and $\mathcal F$\, corresponds to the $A$-module $M$\,, we have 
$$\Gamma(\Spec A, \mathcal O_F)=\text{Sym}^{\bullet}M\quad  \text{and}\quad \pi^*S^N\mathcal F\mid_{\Spec A}=S^NM\otimes_A\text{Sym}^{\bullet}M$$ and the required map is simply the tensor multiplication map
$S^NM\otimes_A\text{Sym}^{\bullet}M\longrightarrow \text{Sym}^{\bullet}M.$\\
In our case, $\mathcal F=\mathcal T^1(X/k)$\, and $ T^{1*}(X/k):=\Spec_X\text{Sym}^{\bullet}\mathcal T^1_{X/k}$\, is the classical cotangent bundle of $X$ with projection $\pi_X: T^{1*}(X/k)\longrightarrow X$\,. Put $T'(X/S):=T^{1,*}(X/S)\backslash s_0(X)$\, where $s_0: X\hookrightarrow T^{1,*}(X/k)$\, is the zero section with same projection $\pi_X: T'(X/k)\longrightarrow X$\,.\\
 We have the following classical
\begin{definition}\mylabel{def:D2512}(Definition of an elliptic operator in algebraic geometry)\\
 With notation as  just introduced, let $X/k$\, be a smooth scheme of finite type over the base field $k$\, , $\mathcal E_1,\mathcal E_2$\, be locally free sheaves on $X$ and $D: \mathcal E_1\longrightarrow \mathcal E_2$\, be a differential operator of order $N$. The operator $D$ is called elliptic, if 
 $$c_{N,\mathcal T^1}\circ \pi_X^*(\sigma^N(D)): \pi_X^*\mathcal E_1\longrightarrow \pi_X^*\mathcal E_2$$ is an isomorphism of locally free sheaves on $T'(X/S)$\,.
\end{definition}
This is just an adoptation of the classical definition of an elliptic operator in differential geometry to the setting of locally free sheaves. Observe that $\mathcal E_1$\, and $\mathcal E_2$\, must then  have the same rank.   
\begin{proposition}\mylabel{prop:P13114} Let $X/k$\, be a smooth complete variety such that\\
 $\mu_{min}(\Omega^{(1)}(X/k))>0$\, with respect to some polarization $H$\,. Then, there is no locally free sheaf $\mathcal E$\, plus an elliptic differential operator $D: \mathcal E\longrightarrow \mathcal E$\, on $X$.
\end{proposition}
\begin{proof} The proof is relatively straight forward. Let $\mathcal E_1$\, be the maximal destabilizing subsheaf. By \prettyref{prop:P60} , $D$ respects the Harder Narasimhan-filtration and $D|_{\mathcal E_1}$\, is $\mathcal O_X$-linear.  The sheaves $\mathcal E_1$\, and $\mathcal E/\mathcal E_1$\, are torsion free and are locally free outside a codimension $2$ subset of $X$. So there exists a Zariski-open $\Spec A\subset X$\, such that $\mathcal E$\, and $\mathcal E_1$\, and $\mathcal E/\mathcal E_1$\,  are free on $X$. Furthermore, by shrinking $\Spec A$\,, we may assume that $\mathcal T^1(X/k)$\, is free on $\Spec A$\,. Choose a trivialization 
\begin{gather*}\mathcal E|_{\Spec A}\cong A^{\oplus n}\cong A^{\oplus k}\oplus A^{\oplus l} \,\,\text{with}\,\,\\
\mathcal E_1|_{\Spec A}\cong A^{\oplus k}\,\,\text{and}\,\, \mathcal E/\mathcal E_1|_{\Spec A}\cong A^{\oplus l}.
\end{gather*}
The symbol $\sigma(D)$\, is given by an $n\times n$-matrix of $N^{th}$-order differential operators on $\Spec A$\,. But  the $(k\times k)$-submatrix that corresponds to $\sigma(D)|_{\mathcal E_1}$\, is the zero matix since $D$ restricted to $\mathcal E_1$\, is $\mathcal O_X$-linear. By definition of ellipticity, the symbol $\sigma(D)$\, can never be a fibrewise isomorphism outside the zero section.
\end{proof}
\begin{remark}This just says, that for "most" varieties, e.g., if $\Omega^{(1)}(X/k)$\, is ample, elliptic operators on locally free sheaves do not exist. So there is no hope to prove the existence of an algebraic Atiyah-Singer-Index theorem that is equivalent to the Hirzebruch-Riemann-Roch theorem.
\end{remark}
In order to give an example of an ellitic operator on a smooth complete variety, we prove 
\begin{proposition}\mylabel{prop:P26113} Let $A/\mathbb C$\, be an abelian variety, or, more generally a complex torus. Then, there exist elliptic operators $\mathcal O_A\longrightarrow \mathcal O_A$\, of arbitrary high order.
\end{proposition}
\begin{proof} Let $A=\mathbb C^n/\Lambda$\, with $\Lambda$\, generated by the vectors $\lambda_j+i\mu_j, j=1,...,2n$\,. Then, if $D$ is a differential operator on $\mathbb C^n$\, with constant coefficents, then $D$ is obviously translation invariant, since $d^1z_i=d^1(z_i+c)$\, for all $c\in \mathbb C$\, and thus $D$ descends to a differential operator $D_A: \mathcal O_A\longrightarrow \mathcal O_A$\,.\\
Now, on $\mathbb C^n$\, on can construct for arbitrary $N>>0$\, elliptic operators with constant coefficients. 
\end{proof} 
\bibliography{GlobalDiffops}

\begin{thebibliography}{1}

\bibitem{Guenther2}
Stefan G{\"u}nther.
\newblock Basics of jet modules and algebraic linear differential operators.
\newblock {\em arXiv.org, math[AG]}, 2018.

\bibitem{Lazarsfeld}
R.~Lazarsfeld.
\newblock {\em Positivity in {A}lgebraic {G}eometry I,II}.
\newblock Springer-Verlag Berlin Heidelberg, 2004.

\bibitem{Huybrechts-Lehn}
Daniel Huybrechts {,}~Manfred Lehn.
\newblock {\em The geometry of {M}oduli {S}paces of {S}heaves}.
\newblock Aspects of {M}athematics {,} {A} {P}ublication of the
  {M}ax-{P}lanck-{I}nstitut f{\"u}r {M}athematik, {B}onn, 1997.

\bibitem{Miyaoka}
Yoichi Miyaoka.
\newblock {\em Geometry of {H}igher {D}imensional {A}lgebraic {V}arieties}.
\newblock Birkh{\"a}user Verlag Basel, Boston, Berlin, 1991.

\bibitem{Indexsatz}
NGuYen Le~Dang Thi.
\newblock Algebraic {A}tiyah-{S}inger-{I}ndex t{h}eorem.
\newblock {\em arXiv:1702.02625v2 [math.KT]}, 2017.

\end{thebibliography}
\bibliographystyle{plain}
\noindent
\emph{E-Mail-adress:}nplc@online.de
\end{document}